\documentclass[11pt]{amsart}
\usepackage{geometry}               
\usepackage{graphicx}
\usepackage{tikz}
\usepackage{bbm}
\usepackage{amssymb}
\usepackage{pdfsync}
\usepackage{hyperref}
\usepackage{verbatim} % this enables the {comment} environment
\usepackage{epstopdf}
\usepackage{mathtools}
\mathtoolsset {showonlyrefs=true} %only referenced equations are numbered
\usepackage{mathabx}
\def\N{\mathbb{N}}

\newcommand{\pvector}[1]{
  \begin{pmatrix}
    #1
  \end{pmatrix}} 
\newcommand{\ddirac}[1]{
  \,\boldsymbol{\delta}\!\pvector{#1}\!} 
\renewcommand{\d}{\,{\rm d}} 

\newcommand{\verteq}{\rotatebox{90}{$=$}}
\newcommand{\verteqspaced}{\rotatebox{90}{$=\ \,$}}
\newcommand{\vertbigspaced}{\rotatebox{90}{$<\ \,$}}

\newtheorem{theorem}{Theorem}
\newtheorem{corollary}[theorem]{Corollary}

\newtheorem{lemma}[theorem]{Lemma}

\title[Sharp Agmon--H\"ormander]{Intermittent symmetry breaking and stability of the sharp Agmon--Hörmander estimate on the sphere}

\author{Giuseppe Negro}
\author{Diogo Oliveira e Silva}
\address{ Departamento de Matemática\\ 
Instituto Superior Técnico\\
Av. Rovisco Pais\\ 
1049-001 Lisboa, Portugal.}
\email{giuseppe.negro@tecnico.ulisboa.pt}
\email{diogo.oliveira.e.silva@tecnico.ulisboa.pt}
                                        
\begin{document}


\begin{abstract}
    We compute the optimal constant and characterise the maximisers at all spatial scales for the Agmon--H\"ormander $L^2$-Fourier adjoint restriction estimate on the sphere. The maximisers switch back and forth from being constants to being non-symmetric at the zeros of two Bessel functions. We also study the stability of this estimate and establish a sharpened version in the spirit of Bianchi--Egnell. The corresponding stability constant and maximisers again exhibit a curious intermittent behaviour.
\end{abstract}

\numberwithin{equation}{section}

\maketitle

\section{Introduction}
In~\cite[Theorem 2.1]{AgHo76}, Agmon and H\"ormander established an estimate for compact manifolds, which reads as follows when applied to the unit sphere $\mathbb S^{d-1}\subset\mathbb R^d$, for $d\ge 2$:
    \begin{equation}\label{eq:AgHor}
        \frac1{\rho} \int_{B_\rho} \lvert \widehat{f\sigma}(x)\rvert^2\, \frac{\d x}{(2\pi)^d} \le C_d(\rho) \int_{\mathbb S^{d-1}} \lvert f(\omega)\rvert^2\,\d\sigma(\omega). 
    \end{equation}
Here, $\sigma$ denotes the standard surface measure on $\mathbb S^{d-1}$ and $B_\rho\subset \mathbb R^d$ denotes the ball of radius $\rho>0$ centred at the origin. Also, $\widehat{f\sigma}$ denotes the Fourier transform; see the forthcoming~\eqref{eq:HeckeBochner}.

In the aforementioned paper, it is shown that the quantity $C_d(\rho)$ is uniformly bounded in $\rho$, but no explicit value is given for it. The first purpose of this note is to provide a proof of~\eqref{eq:AgHor} that yields the optimal value $\mathbf{C}_d(\rho)$, as well as the functions that attain it. We will see that, as $\rho$ increases, these maximising functions change intermittently at the zeros of the Bessel functions $J_{\nu}(\rho)$ and $J_{\nu +1}(\rho)$, where $\nu=\tfrac{d}2-1$; recall (e.g.\@ from Appendix \ref{app:Bessel}) that these functions do not have common positive zeros. For $k\in \mathbb N_{\ge 0}$, we introduce the function 
\begin{equation}\label{eq:Lambda_Intro}
    \Lambda_{k, d}(\rho):=\frac \rho 2 J_{\nu +k}^2(\rho) - \frac \rho 2 (J_{\nu + k-1}J_{\nu +k +1})(\rho).
\end{equation}
Finally, we let $\mathcal{H}_k$ denote the vector subspace of $L^2(\mathbb{S}^{d-1})$ consisting of the spherical harmonics of degree $k\in \mathbb N_{\ge 0}$; in particular, $\mathcal H_0$ consists of the constants. See \S\ref{sec:ProofThm1} for details.
\begin{theorem}\label{thm:main}
   For each $\rho>0$, the optimal constant $\mathbf{C}_d(\rho)$ equals
   \begin{equation}\label{eq:optimal_constant}
        \max_{0\neq f\in L^2(\mathbb S^{d-1})} \frac{\displaystyle \frac1\rho\int_{B_\rho} \lvert \widehat{f\sigma}(x)\rvert^2\, \frac{\d x}{(2\pi)^d}}{\lVert f\rVert_{L^2(\mathbb S^{d-1})}^2}=
         \begin{cases}
            \displaystyle
            \Lambda_{0, d}(\rho) , & (J_{\nu}J_{\nu+1})(\rho)\ge  0,\\
            \displaystyle
            \Lambda_{1, d}(\rho), & (J_{\nu}J_{\nu +1})(\rho)\le 0.\\
        \end{cases}
    \end{equation}    
    The maximum is attained if and only if $f\in \mathcal{M}_d(\rho)\setminus\{0\}$, where $\mathcal{M}_d(\rho)$ equals
    
    \begin{tabular}{clll}
        \emph{(i)} & 
            $\mathcal H_0$,& $\qquad$ & $(J_{\nu}J_{\nu +1})(\rho)>0$\,\emph{;} \\
            \emph{(ii)} &
            $\displaystyle \mathcal H_1$, & &  $(J_{\nu}J_{\nu +1})(\rho)<0$\,\emph{;}\\
            \emph{(iii)} & 
            $\displaystyle \mathcal H_0\oplus \mathcal H_1$, & & $J_{\nu}(\rho)=0$\,\emph{;} \\
            \emph{(iv)} & $\displaystyle \mathcal H_0\oplus \mathcal H_1 \oplus \mathcal H_2$, &  &$J_{\nu+1}(\rho)=0$.
    \end{tabular}
\end{theorem}
\begin{figure}
    \centering
    \includegraphics[width=0.7\textwidth]{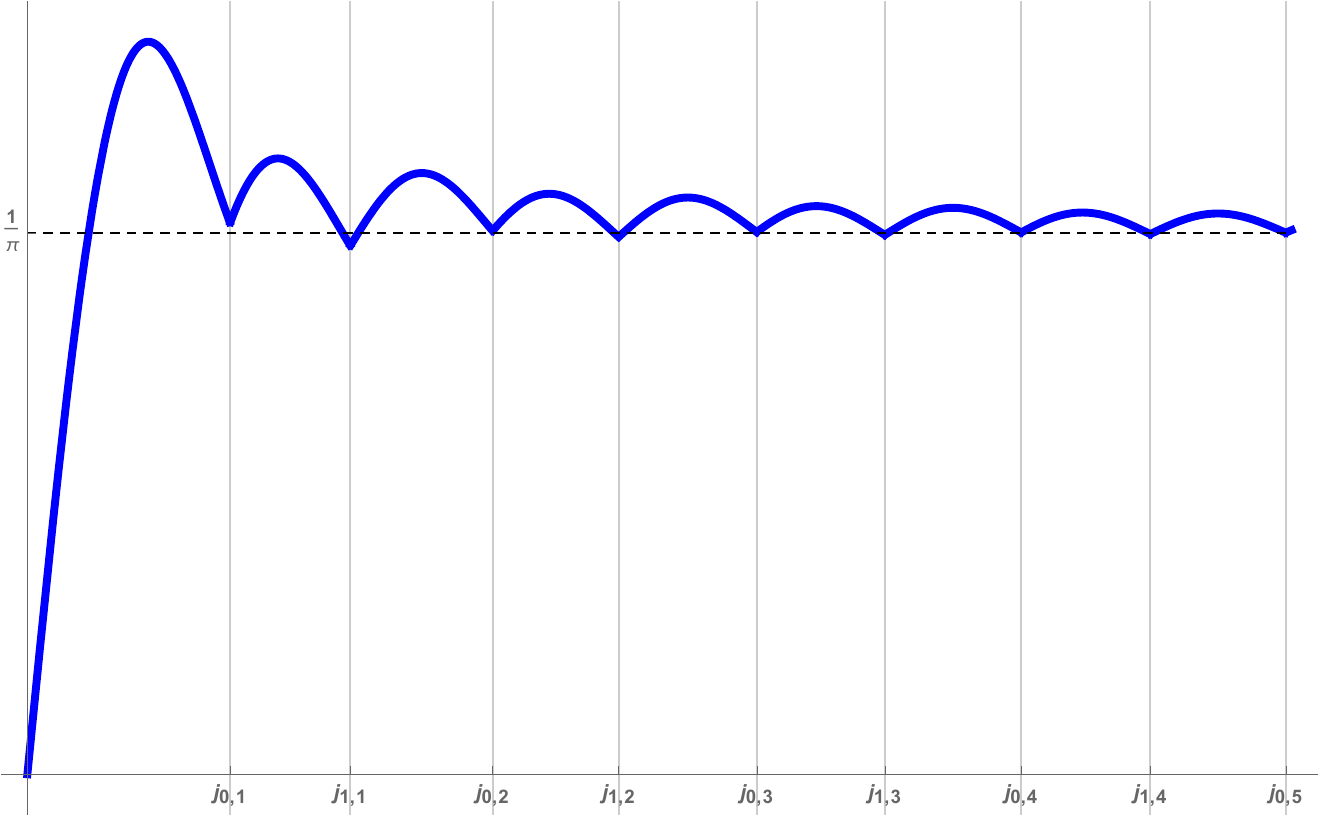}
    \caption{The sharp Agmon--H\"ormander estimate on $\mathbb{S}^1\subset\mathbb R^2$:
    plot of the optimal constant $\mathbf{C}_2(\rho)$, for $0< \rho\leq 15$.}
    \label{fig:CdGraph}
\end{figure}

In case~(ii), the constants $\mathcal H_0$ are not maximisers; this is remarkable, since both sides of~\eqref{eq:AgHor} are rotationally invariant, hence one might expect the maximisers to be invariant functions (the constants being the only such functions on the sphere). In the literature, this phenomenon is sometimes called~\emph{symmetry breaking} (see, e.g.~\cite[p.~25]{Do21}), and it is especially studied in the context of inequalities related  to elliptic PDE. We find it striking that in our context, involving oscillatory operators, symmetry breaking occurs even at the relatively elementary level of~\eqref{eq:AgHor}, which is an $L^2$-estimate and  as such can be studied via orthogonality methods. 

Another noteworthy feature of Theorem~\ref{thm:main} is that the space of maximisers $\mathcal M_d(\rho)$ does not vary continuously with $\rho$, but instead is locally constant, except for a sequence of jumps at the zeroes of $J_\nu$ and $J_{\nu+1}$. This  {lies at the root} of another  {interesting} fact, which follows from the next result. Letting $\d(f, \mathcal{M}_d(\rho))=\inf \{\lVert f-m\rVert_{L^2(\mathbb S^{d-1})}\,:\, m\in \mathcal{M}_d(\rho)\}$, and denoting by 
\begin{equation}\label{eq:delta_deficit}
    \delta_d(f;\rho):= \mathbf{C}_d(\rho)\lVert f\rVert_{L^2(\mathbb S^{d-1})}^2 -\frac{1}{\rho}\int_{B_\rho}\lvert \widehat{f\sigma}(x)\rvert^2\,\frac{\d x}{(2\pi)^d}
\end{equation}
the deficit functional for~\eqref{eq:AgHor}, we establish the following sharpened version in the spirit of Bianchi--Egnell~\cite{BE91}. Note that $\delta_d(f;\rho)\ge 0$, 
with $\delta_d(f;\rho)=0$ if and only if $f\in\mathcal{M}_d(\rho)$.
\begin{theorem}\label{thm:stability}
    For each $\rho>0$, the following inequalities hold, and the corresponding multiplicative constants are optimal:
    \begin{equation}\label{eq:stability}
        \mathbf{S}_d(\rho) \d(f, \mathcal{M}_d(\rho))^2 \le \delta_d(f;\rho)\le \mathbf{C}_d(\rho) \d(f, \mathcal{M}_d(\rho))^2.
    \end{equation}
    There is equality in the right-hand inequality only in the trivial case $f\in \mathcal M_d(\rho)$, and there is equality in the left-hand inequality if and only if $f\in \mathcal{M}_d(\rho)\oplus \mathcal{E}_d(\rho)$, where
    \begin{itemize}
        \item[(i)] for\footnote{The symbol $\verteq$ is to be replaced by either of the values above or the below, which coincide in each case.} $(J_\nu J_{\nu +1})(\rho)>0$, 
        \begin{equation}\label{eq:Scasei}
            \begin{array}{ll|r}
                \mathbf{S}_d(\rho)=
                    \begin{array}{c} 
                        \Lambda_{0, d}(\rho)-\Lambda_{1, d}(\rho), \\
                        \verteq\\
                        \Lambda_{0, d}(\rho)-\Lambda_{2, d}(\rho),\\
                    \end{array}
            & 
                \mathcal{E}_d(\rho)= 
                \begin{array}{l} 
                    \mathcal{H}_1, \\ \mathcal{H}_1\oplus\mathcal{H}_2\oplus\mathcal{H}_3, \\
                    \mathcal{H}_2, 
                \end{array}
            &
                \begin{array}{c}
                    (J_{\nu+1}J_{\nu+2})(\rho)>0, \\   (J_{\nu+1}J_{\nu+2})(\rho)=0 , \\ (J_{\nu+1}J_{\nu+2})(\rho)<0.
                \end{array}
            \end{array}
        \end{equation}   
        \item[(ii)] for $(J_{\nu}J_{\nu+1})(\rho)<0$, letting  $\mathfrak{J}_\nu:=J_{\nu}J_{\nu+1}+J_{\nu+1}J_{\nu+2}+J_{\nu+2}J_{\nu+3}$, 
        \begin{equation}\label{eq:Scaseii}
            \begin{array}{ll|r}
                \mathbf{S}_d(\rho)=
                \begin{array}{c} 
                    \Lambda_{1, d}(\rho)-\Lambda_{0, d}(\rho),\\
                   \verteq\\
                    \Lambda_{1, d}(\rho)-\Lambda_{3, d}(\rho),\\
                \end{array} 
                &
                \mathcal{E}_d(\rho)= 
                \begin{array}{l} 
                    \mathcal H_0,  \\
                    \mathcal{H}_0\oplus \mathcal{H}_3, \\
                    \mathcal H_3,  \\
                \end{array}
            &
                \begin{array}{c}
                    \mathfrak{J}_\nu(\rho)> 0, \\
                    \mathfrak{J}_\nu(\rho)= 0,\\
                    \mathfrak{J}_\nu(\rho)< 0.
                \end{array}
            \end{array}
        \end{equation}   
        \item[(iii)] for $J_\nu(\rho)=0$, 
        \begin{equation}\label{eq:Scaseiii}
            \begin{array}{ll|r}
                \mathbf{S}_d(\rho)=
                \begin{array}{c} 
                    \Lambda_{0, d}(\rho)-\Lambda_{2, d}(\rho),\\
                    \Lambda_{0, d}(\rho)-\Lambda_{3, d}(\rho), \\
                \end{array}
                &
                \mathcal{E}_d(\rho)= 
                \begin{array}{l}
                    \mathcal{H}_2,  \\ \mathcal{H}_3, \\
                \end{array}
                &
                \begin{array}{c}
                     (J_{\nu+2}J_{\nu+3})(\rho)>0, \\
                        (J_{\nu+2}J_{\nu+3})(\rho)<0. 
                \end{array}
            \end{array}
        \end{equation}  
        \item[(iv)] for $J_{\nu +1}(\rho)=0$, 
        \begin{equation}\label{eq:Scaseiv}
            \begin{array}{ll|c}
                \mathbf{S}_d(\rho)=
                \begin{array}{c} 
                    \Lambda_{0, d}(\rho)-\Lambda_{3, d}(\rho),\\
                    \Lambda_{0, d}(\rho)-\Lambda_{4, d}(\rho)\\
                \end{array}
                &
                \mathcal{E}_d(\rho)= 
                \begin{array}{l} 
                    \mathcal{H}_3, \\ \mathcal{H}_4,  \\
                \end{array}
                &
                \begin{array}{c}
                   (J_{\nu+3}J_{\nu+4})(\rho)>0, \\
                 (J_{\nu+3}J_{\nu+4})(\rho)<0.
                \end{array}
            \end{array}
        \end{equation}
    \end{itemize}
\end{theorem}    
\begin{figure}
    \centering
    \includegraphics[width=0.7\textwidth]{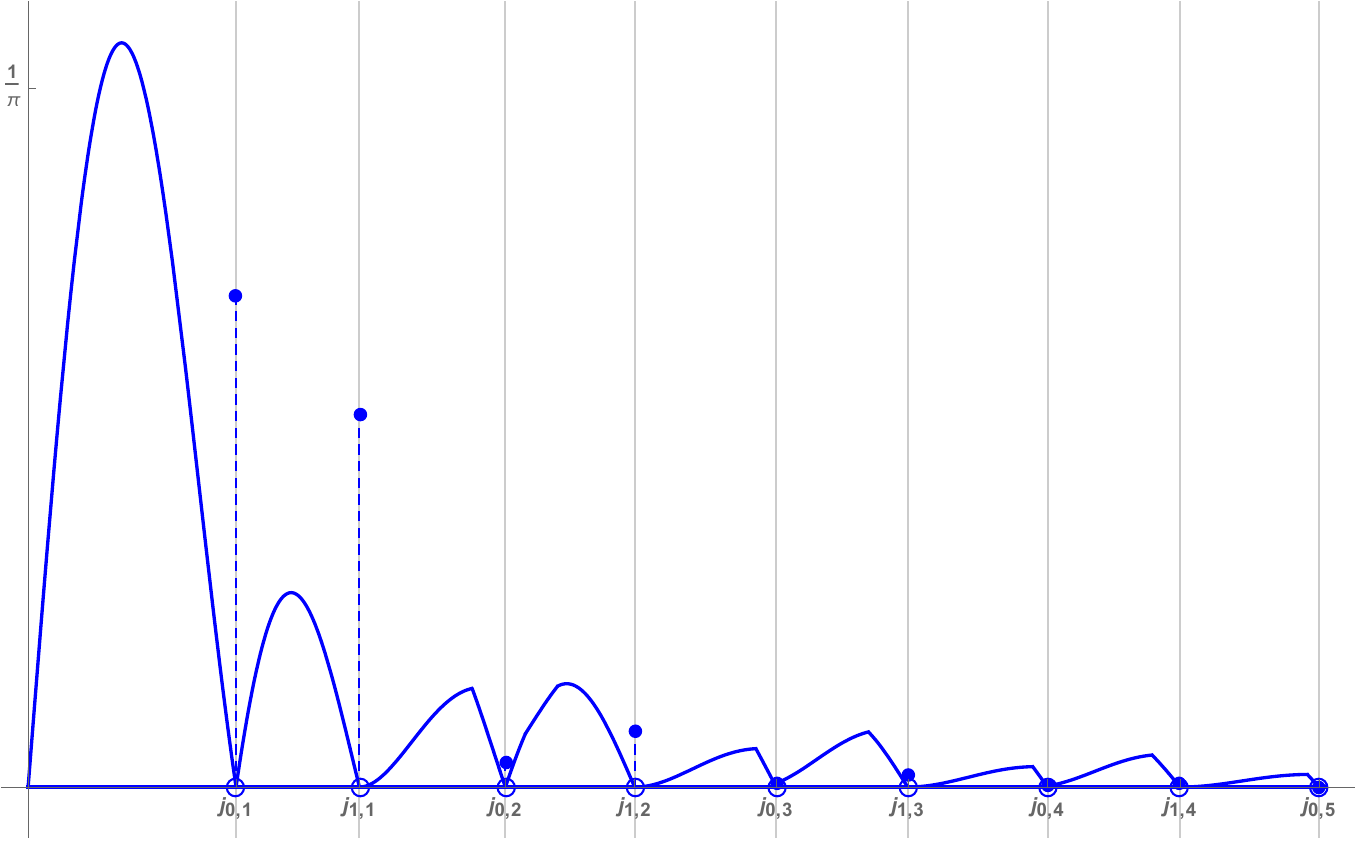}
    \caption{The sharpened Agmon--H\"ormander estimate on $\mathbb{S}^1\subset\mathbb R^2$: 
    plot of the stability constant $\mathbf{S}_2(\rho)$, for $0< \rho\leq 15$.}
    \label{fig:SdGraph}
\end{figure}
The left-hand inequality in Theorem~\ref{thm:stability} quantifies the fact that any function which comes close to attaining the optimal constant in~\eqref{eq:AgHor} must lie within a small distance of the space of maximisers; this phenomenon is called \emph{stability}. We find it remarkable that the stability constant ${\bf S}_d(\rho)$ does not define a continuous function of $\rho$; it exhibits certain jumps, corresponding to the aforementioned jumps of $\mathcal M_d(\rho)$.  This will be proved as part of the following result, which is a study of the regularity of ${\bf C}_d(\rho)$ and ${\bf S}_d(\rho)$, and a consequence of the previous theorems.
\begin{corollary}\label{cor:regularity}
Let $d\geq 2$ and $\rho\in (0,\infty)$. Then:
\begin{itemize}
    \item The function $\rho\mapsto{\bf C}_d(\rho)$ is not differentiable at each positive zero of $J_\nu J_{\nu+1}$. It defines a Lipschitz function on $(0,\infty)$ which is real-analytic between any two consecutive zeroes of $J_\nu J_{\nu+1}$.
    \item The function $\rho\mapsto{\bf S}_d(\rho)$ has a jump discontinuity at each positive zero of $J_\nu J_{\nu+1}$. It defines a piecewise real-analytic function between any two consecutive zeroes of $J_\nu J_{\nu+1}$, which fails to be differentiable at each positive zero of $J_{\nu+2}$. 
\end{itemize}
\end{corollary}
Some of the loss-of-regularity phenomena exhibited by ${\bf C}_d(\rho)$ and ${\bf S}_d(\rho)$ are depicted in Figures \ref{fig:CdGraph} and \ref{fig:SdGraph}, respectively. The behaviour of optimal constants has been studied in the context of the non-oscillatory Brascamp--Lieb inequalities, see \cite{BBFL18,BBCF17}, where it found numerous applications, in particular to Fourier restriction theory.

In the $L^2$-setting of the present paper, estimate \eqref{eq:AgHor} has been generalised  \cite{BRV97}, and sharp and sharpened inequalities in the context of smoothing and trace estimates have been extensively investigated; see \cite{BJOS18, BSS15}, and the references therein.
Within this general framework, the main new contributions of Theorems \ref{thm:main} and \ref{thm:stability} lie in the complete solutions to both the sharp and the sharpened problems, in terms of the explicit values for the optimal constants ${\bf C}_d(\rho),{\bf S}_d(\rho)$ and the full characterisation of the spaces of maximisers $\mathcal M_d, \mathcal E_d$.

We finish the Introduction with some brief remarks on {\it sharp Fourier restricion theory}.
Inequality~\eqref{eq:AgHor} can be regarded as the most basic example of an adjoint Fourier restriction estimate on the sphere \cite[\S 5]{Ta03}. The corresponding optimal constant and maximisers are known only for a few such estimates; most notably, Foschi~\cite{Fo15} proved that constant functions maximise the endpoint Stein--Tomas estimate on $\mathbb S^2$. It turns out that constant functions maximise the $L^2(\mathbb S^{d-1})-L^{2n}(\mathbb R^d)$  adjoint restriction estimate for every $d\in\{3, 4, 5, 6, 7\}$ and integer $n\ge 2$; see~\cite{CaOl15, OlQu21}.
Moreover, sharpened Fourier restriction inequalities have been recently established in \cite{Go19, Go19b, GoNe20, GoZa20, Ne18}.
We refer the interested reader to the survey \cite{FoOl17} for a more extended discussion and further references.
\vspace{.1cm}

{\it Structure of the paper.}
We prove Theorem \ref{thm:main}, Theorem \ref{thm:stability}, and Corollary \ref{cor:regularity} in \S \ref{sec:ProofThm1}, \S \ref{sec:stability}, and \S \ref{sec:regularity}, respectively.
For every $f\in L^2(\mathbb S^{d-1})$, our methods also yield the limit
\begin{equation}\label{eq:AgHorLimit}
    \lim_{\rho\to \infty} \frac1\rho\int_{B_\rho}\lvert \widehat{f\sigma}(x)\rvert^2\, \frac{\d x}{(2\pi)^d}
    = \frac1\pi \int_{\mathbb S^{d-1}} \lvert f(\omega)\rvert^2\, \d\sigma(\omega), 
\end{equation}
in accordance with Agmon--H\"ormander~\cite[Theorem~3.1]{AgHo76}. We discuss this limit in \S \ref{sec:limit}. An immediate corollary is that  $\mathbf C_d(\rho)\to \tfrac 1 \pi$ and $\mathbf{S}_d(\rho)\to 0$, as $\rho \to \infty$; see Figures~\ref{fig:CdGraph} and \ref{fig:SdGraph}.
Finally, we collect all the relevant facts concerning Bessel functions in Appendix \ref{app:Bessel}.

\section{The optimal constant $\mathbf{C}_d(\rho)$ and its maximisers: Proof of Theorem~\ref{thm:main}}
\label{sec:ProofThm1}
Throughout the paper, set $\nu=d/2 -1$. We will always use the notation $Y_k(\omega)$ to denote a spherical harmonic of degree $k\in\mathbb N_{\ge 0}$, which by definition is a complex-valued homogeneous harmonic polynomial in $\omega=(\omega_1, \ldots, \omega_d) \in\mathbb R^d$, of degree $k$, considered as a function on $\mathbb S^{d-1}$. In particular, $Y_k$ is never the zero function. As stated in the Introduction, we denote
\begin{equation}\label{eq:mathcalH_precise}
    \mathcal{H}_k:=\{Y_k\,:\, Y_k \text{ is a spherical harmonic of degree }k\}\cup\{0\},
\end{equation}
which is a finite-dimensional vector subspace of $L^2(\mathbb S^{d-1})$. Spherical harmonics of different degrees are mutually orthogonal and form a complete system, see \cite[Chapter IV, Corollary 2.3]{StWe}, meaning that to each nonzero $f\in L^2(\mathbb S^{d-1})$ there uniquely correspond $F[f]\subseteq \mathbb N_{\ge 0}$ and {$\{Y_k[f]\in\mathcal{H}_k\setminus\{0\}\,:\,k\in F[f]\}$} such that 
\begin{equation}\label{eq:spherical_decomp}
    \begin{array}{cc}
    \displaystyle
        f=\sum_{k\in F[f]} Y_k[f], & 
        \displaystyle\text{ thus } \lVert f\rVert_{L^2(\mathbb S^{d-1})}^2 =\sum_{k\in F[f]} \lVert Y_k[f]\rVert_{L^2(\mathbb S^{d-1})}^2.
    \end{array}
\end{equation}
For notational convenience, we will leave out the dependence on $f$, writing $F$ and $Y_k$ in place of $F[f]$ and $Y_k[f]$, respectively. 

For a single spherical harmonic, we have the Fourier transform formula
\begin{equation}\label{eq:HeckeBochner}
     \begin{array}{cc}\displaystyle
          \widehat{Y_k\sigma}(\xi)=\int_{\mathbb S^{d-1}} Y_k(\omega)e^{-i\omega\cdot \xi}\, \d\sigma(\omega)=\frac{(2\pi)^{\frac d 2}}{i^{k}}\frac{J_{\nu+k}(\lvert\xi\rvert)}{\lvert \xi\rvert^\nu} Y_k\left(\frac{\xi}{\lvert\xi\rvert}\right),&
          \xi\in\mathbb R^d;
    \end{array}
\end{equation}
see, for example,~\cite[Chapter~IV, Theorem 3.10]{StWe}.  
 Using the decomposition~\eqref{eq:spherical_decomp} and integrating in polar coordinates, we obtain the diagonal form of the left-hand side in \eqref{eq:AgHor}:
\begin{equation}\label{eq:AgHor_decomp}
    \begin{split}
        \frac{1}{\rho}\int_{B_\rho} \lvert \widehat{f\sigma}(x)\rvert^2\,\frac{\d x}{(2\pi)^d}&=\sum_{k\in F} \frac1\rho \int_0^\rho \int_{\mathbb S^{d-1}}\lvert \widehat{Y_k\sigma}(r\omega)\rvert^2\, r^{d-1}\frac{\d r\d\sigma(\omega)}{(2\pi)^d} \\
        &=\sum_{k\in F} \frac{1}\rho\left(\int_0^\rho J_{\nu+k}^2(r) r\d r\right)\lVert Y_k\rVert_{L^2(\mathbb S^{d-1})}^2.
    \end{split}
\end{equation}
As we will see in Lemma~\ref{lem:MainLemma} below, the latter integral can be evaluated explicitly, yielding 
\begin{equation}\label{eq:AgHor_diag}
    \frac{1}{\rho}\int_{B_\rho} \lvert \widehat{f\sigma}(x)\rvert^2\,\frac{\d x}{(2\pi)^d}=\sum_{k\in F} \Lambda_{k, d}(\rho)\lVert Y_k\rVert_{L^2(\mathbb S^{d-1})}^2, 
\end{equation}
where the coefficients $\Lambda_{k, d}$ have been introduced in~\eqref{eq:Lambda_Intro}.
In turn, the coefficients $\Lambda_{k, d}$ are related to the optimal constant $\mathbf{C}_d(\rho)$ via the following simple observation.

\begin{lemma}\label{lem:Rayleigh_quotient}
    For each $\rho>0$, 
    \begin{equation}\label{eq:Cd_optimal_bound}
        \mathbf{C}_d(\rho):=\sup_{0\ne f\in L^2(\mathbb S^{d-1})} \frac{\displaystyle\frac1 \rho\int_{B_\rho} \lvert \widehat{f\sigma}(x)\rvert^2\, \frac{\d x}{(2\pi)^d}}{\lVert f \rVert_{L^2(\mathbb S^{d-1})}^2} = \sup\{ \Lambda_{k, d}(\rho)\ :\ k\in\N_{\geq 0} \}.
    \end{equation}
    Letting $K:=\{k\in \mathbb N_{\ge 0}\, :\, \Lambda_{k, d}(\rho)=\sup_{h\in\N_{\geq 0}}  \Lambda_{h, d}(\rho)\}$, we have that $f\in L^2(\mathbb S^{d-1})\setminus\{0\}$ attains the supremum in~\eqref{eq:Cd_optimal_bound} if and only if
    \begin{equation}\label{eq:lin_comb_sph_harm}
        \begin{array}{cc}
            f(\omega)=Y_{k_1}(\omega)+Y_{k_2}(\omega)+\ldots+Y_{k_n}(\omega), & k_j\in K.
        \end{array}
    \end{equation}
\end{lemma}
We remark that, a priori, $K$ could be empty, or it could be infinite; we will prove that this is never the case.
\begin{proof}[Proof of Lemma~\ref{lem:Rayleigh_quotient}]
From \eqref{eq:spherical_decomp} and \eqref{eq:AgHor_diag} it follows that 
    \begin{equation}\label{eq:Rayleigh}
        \frac{\displaystyle\frac1 \rho\int_{B_\rho} \lvert \widehat{f\sigma}(x)\rvert^2\, \frac{\d x}{(2\pi)^d}}{\lVert f \rVert_{L^2  {(\mathbb S^{d-1})}}^2} = \frac{\displaystyle \sum_{k\in F} \Lambda_{k, d}(\rho)\lVert Y_k\rVert_{L^2  {(\mathbb S^{d-1})}}^2}{\displaystyle \sum_{k\in F} \lVert Y_k\rVert_{L^2  {(\mathbb S^{d-1})}}^2}.
    \end{equation}
    It is clear that the supremum of this quotient is $\sup_{k\ge 0} \Lambda_{k, d}(\rho)$, with equality if and only if $F\subseteq K$, verifying~\eqref{eq:lin_comb_sph_harm} and concluding the proof of the lemma.
\end{proof}
We now record some useful properties of the coefficients $\Lambda_{k, d}(\rho)$.   
\begin{lemma}\label{lem:MainLemma}
    For each $\rho>0$ and $k\in\N_{\geq 0}$, 
    \begin{equation}\label{eq:Lommel}
        \frac1\rho\int_0^\rho J_{\nu+k}^2(r)r\d r=\Lambda_{k, d}(\rho)= \frac \rho2 J_{\nu+k}^2(\rho)-\frac{\rho}2 J_{\nu+k-1}(\rho)J_{\nu+k+1}(\rho).
    \end{equation}
    Moreover,
    \begin{align}\label{eq:recursion_odd}
        \Lambda_{k, d}(\rho)-\Lambda_{k+1, d}(\rho)&=J_{\nu+k}(\rho)J_{\nu+k+1}(\rho), \\
        \label{eq:recursion_even}
        \Lambda_{k, d}(\rho)-\Lambda_{k+2, d}(\rho)&=\frac{2(\nu+k+1)}{\rho}J_{\nu+k+1}^2(\rho).
    \end{align}
    In particular, $\Lambda_{k, d}(\rho)\ge \Lambda_{k+2, d}(\rho)$, with equality if and only if $J_{\nu+k+1}(\rho)=0$.
\end{lemma}
\begin{proof}
    Identity~\eqref{eq:Lommel} is due to Lommel; see~\cite[\S 5.11 (11)]{Wa44}. We apply this identity, together with the Bessel recursion~\eqref{eq:RR2} in Appendix~\ref{app:Bessel}, to obtain
    \begin{equation}\label{eq:LambdaZero_vs_LambdaOne}
        \begin{split}
            &\Lambda_{k, d}(\rho)-\Lambda_{k+1,d}(\rho) = \\  &=\frac \rho 2 \left[
                J_{\nu+k}(\rho)\left( J_{\nu+k}(\rho)+J_{\nu+k+2}(\rho)\right) -J_{\nu+k+1}(\rho)\left(J_{\nu+k-1}(\rho)+J_{\nu+k+1}(\rho)\right)\right] \\ 
            &=(\nu+k+1)(J_{\nu+k} J_{\nu+k+1})(\rho) - (\nu+k) (J_{\nu+k}J_{\nu+k+1})(\rho)   \\     
            &=(J_{\nu+k} J_{\nu+k+1})(\rho),
        \end{split}
    \end{equation}
    proving~\eqref{eq:recursion_odd}. We could use~\eqref{eq:recursion_odd} to prove~\eqref{eq:recursion_even}, but the latter can also be seen to follow directly from the integral formula~\eqref{eq:Lommel} for $\Lambda_{k, d}$; indeed, applying the Bessel   recursions~\eqref{eq:RR2}--\eqref{eq:RR1}, we have that
    \begin{equation}\label{eq:MainLemmaProof}
        \begin{split}
            \Lambda_{k, d}(\rho)-\Lambda_{k+2, d}(\rho) &=\frac1\rho\int_0^\rho (J^2_{\nu+k}(r)-J^2_{\nu+k+2}(r))r\, \d r \\
            &=\frac{4(\nu+k+1)}{\rho}\int_0^{\rho}J_{\nu+k+1}(r)J'_{\nu+k+1}(r)\, \d r \\
            &=\frac{2(\nu+k+1)}{\rho}J^2_{\nu+k+1}(\rho)  {,}
        \end{split}
    \end{equation}
    where in the last computation we used the fact that $J_{\nu+k+1}(0)=0$.
\end{proof}
Having settled these classical preliminaries, we now start with the actual proof of Theorem~\ref{thm:main}. 
By Lemma \ref{lem:MainLemma} we obtain the following chains of inequalities.
\begin{itemize}
    \item[(i)]  {{\it Case}} $(J_\nu J_{\nu+1})(\rho)>0$. We have 
    \begin{equation}\label{eq:big_chain_i}
        \begin{split}
            \Lambda_{0, d}(\rho)>\Lambda_{1, d}(\rho)\ge \Lambda_{3, d}(\rho)\ge \Lambda_{5, d}(\rho)\ge \ldots \\
            \Lambda_{0, d}(\rho)>\Lambda_{2, d}(\rho)\ge \Lambda_{4, d}(\rho)\ge \Lambda_{6, d}(\rho)\ge \ldots
        \end{split}
    \end{equation}
    The strict inequalities follow from~\eqref{eq:recursion_odd} and~\eqref{eq:recursion_even}, respectively.
    \item[(ii)]  {{\it Case}} $(J_{\nu}J_{\nu+1})(\rho)<0$. We have 
    \begin{equation}\label{eq:big_chain_ii}
        \begin{split}
            \Lambda_{1, d}(\rho)>\Lambda_{0, d}(\rho)\ge \Lambda_{2, d}(\rho)\ge \Lambda_{4, d}(\rho)\ge \ldots\\
            \Lambda_{1, d}(\rho)>\Lambda_{3, d}(\rho)\ge \Lambda_{5, d}(\rho)\ge \Lambda_{7, d}(\rho)\ge \ldots
        \end{split}
    \end{equation}
    The strict inequalities are obtained as before, noting that $J_{\nu+2}(\rho)\ne 0$; see Lemma~\ref{lem:bessel_lemma}.
    \item[(iii)]  {{\it Case}} $J_{\nu}(\rho)=0$.  Note that $J_{\nu+1}(\rho)\ne0$ and $J_{\nu+2}(\rho)\ne 0$, by Lemma \ref{lem:Bourget}. Reasoning as in the previous steps, we obtain
    \begin{equation}\label{eq:big_chain_iii}
        \begin{array}{cc}
            \Lambda_{0, d}(\rho)\!\!\!\! & >\Lambda_{2, d}(\rho)\ge \Lambda_{4, d}(\rho)\ge \Lambda_{6, d}(\rho)\ge\ldots \\
            \verteqspaced \!\!\!\!& \\
            \Lambda_{1,d}(\rho)\!\!\!\!& >\Lambda_{3,d}(\rho) \ge \Lambda_{5, d}(\rho) \ge \Lambda_{7, d}(\rho)\ge\ldots
        \end{array}
    \end{equation}
    \item[(iv)]  {{\it Case}} $J_{\nu+1}(\rho)= 0$.  {Since} $J_{\nu+2}(\rho)\ne  0  {\ne J_{\nu+3}(\rho)}$, we have
    \begin{equation}\label{eq:big_chain_iv}
        \begin{array}{cc}
            \Lambda_{0, d}(\rho)\!\!\!\! & =\Lambda_{2, d}(\rho) > \Lambda_{4, d}(\rho)\ge \Lambda_{6, d}(\rho)\ge \ldots \\
            \verteqspaced\!\!\!\! & \\
            \Lambda_{1,d}(\rho)\!\!\!\!& >\Lambda_{3,d}(\rho) \ge \Lambda_{5, d}(\rho) \ge \Lambda_{7, d}(\rho)\ge \ldots
        \end{array}
    \end{equation}
\end{itemize}
We conclude that $\mathbf{C}_d(\rho)=\Lambda_{0, d}(\rho)$ in Case (i), and $\mathbf{C}_d(\rho)=\Lambda_{1, d}(\rho)$ in Case (ii), while $\mathbf{C}_d(\rho)=\Lambda_{k, d}(\rho)$ for $k\in\{0, 1\}$ in Case (iii), and  for $k\in\{0,1,2\}$ in Case (iv). The optimal constant is attained by single spherical harmonics of degree $0$ or $1$ in Cases (i) and (ii), respectively, and it is attained by linear combinations of spherical harmonics of degrees $\{0, 1\}$ or $\{0, 1, 2\}$ in Cases (iii) and (iv), respectively. The proof of Theorem~\ref{thm:main} is complete. 

\section{ Sharp stability: Proof of Theorem~\ref{thm:stability} }\label{sec:stability}
Here we prove Theorem \ref{thm:stability}.
Reasoning as in the previous section, we write the deficit functional as follows:
\begin{equation}\label{eq:deficit}
    \begin{split}
        \delta_d(f;\rho)&=\mathbf{C}_d(\rho)\lVert f\rVert_{L^2(\mathbb S^{d-1})}^2 -\frac{1}{\rho}\int_{ {B_\rho}} \lvert \widehat{f\sigma}(x)\rvert^2\,\frac{\d  x}{(2\pi)^d} \\ 
        &= \sum_{k\,:\,\Lambda_{k, d}(\rho)<\mathbf{C}_d(\rho)}  (\mathbf C_d(\rho) -\Lambda_{k, d}(\rho)) \lVert Y_k\rVert_{L^2(\mathbb S^{d-1})}^2,
    \end{split}
\end{equation}
and we notice that 
\begin{equation}\label{eq:distance_is_sum_squares}
    \sum_{k\,:\,\Lambda_{k, d}(\rho)<\mathbf{C}_d(\rho)}\lVert Y_k\rVert_{L^2(\mathbb S^{d-1})}^2 = \d(f, \mathcal{M}_d(\rho))^2.
\end{equation}
We immediately infer the inequalities 
\begin{equation}\label{eq:stability_primordial}
    \mathbf{S}_d(\rho)\d (f, \mathcal M_d(\rho))^2 \le \delta_d(f; \rho)\le {\tilde{\mathbf{C}}}_d(\rho)\d (f, \mathcal M_d(\rho))^2,
\end{equation}
where, letting $A:=\{ \mathbf{C}_d(\rho)-\Lambda_{k, d}(\rho)\, :\, k\in \mathbb{N}_{\ge0}\text{ such that }\Lambda_{k, d}(\rho)<\mathbf{C}_d(\rho)\},$
\begin{equation}\label{eq:S_d_primordial}
    \begin{array}{cc}
        \mathbf{S}_d(\rho):=\inf A, & \tilde{\mathbf{C}}_d {(\rho)}:=\sup A.
    \end{array}
\end{equation}
By the Bessel limit~\eqref{eq:crude_bound}, we see that $\Lambda_{k, d}(\rho)\to 0$ as $k\to \infty$, for all $\rho>0$. In particular, we infer that $ \tilde{\mathbf{C}}_d(\rho) = \mathbf{C}_d(\rho)$, and that the right-hand inequality in~\eqref{eq:stability_primordial} is always strict, except in the trivial case  {when} $\d(f, \mathcal{M}_d(\rho))=0$,  {i.e.}\@ $f\in \mathcal{M}_d(\rho)$. This completes the proof of the right-hand inequality in Theorem~\ref{thm:stability}.

To complete the proof of the left-hand inequality, we need to compute $\mathbf{S}_d(\rho)$ and the set of indices $k\in \mathbb N_{\ge 0}$ that attain {the infimum in~\eqref{eq:S_d_primordial}}; the space $\mathcal{E}_d(\rho)$ will then coincide with the direct sum of the corresponding $\mathcal H_k$. In the following, we will repeatedly  {make} use  {of} Lemma~\ref{lem:MainLemma}, and the index $h$ will always range over $\mathbb N_{\ge 0}$.

     {\it{Case}} (i): $(J_\nu J_{\nu +1})(\rho)>0$. We have $\mathbf{C}_d(\rho)=\Lambda_{0, d}(\rho)$. By the inequalities~\eqref{eq:big_chain_i}, 
    \begin{equation}\label{eq:S_d_case_i}
        \mathbf{S}_d(\rho)=\Lambda_{0, d}(\rho)-\max_{k\,\in\{1, 2\}} \Lambda_{k, d}(\rho).
    \end{equation}
    Now we recall that 
    \begin{equation}\label{eq:StabCaseiSelector}
        \Lambda_{1, d}(\rho)-\Lambda_{2, d}(\rho)=(J_{\nu +1}J_{\nu +2})(\rho).
    \end{equation}
    So, if $(J_{\nu+1}J_{\nu+2})(\rho)>0$,  then  $\mathbf{S}_d(\rho)=\Lambda_{0, d}(\rho)-\Lambda_{1, d}(\rho)$. Moreover, by~\eqref{eq:recursion_even} we see that $\Lambda_{1, d}(\rho)>\Lambda_{3, d}(\rho)\ge \Lambda_{3+2h, d}(\rho)$, thus  $\mathcal{E}_d(\rho)=\mathcal H_1$. 
    
    On the other hand, if $(J_{\nu +1}J_{\nu+2})(\rho)<0$, then $\mathbf{S}_d(\rho)=\Lambda_{0, d}(\rho)-\Lambda_{2, d}(\rho),$ and $J_{\nu+3}(\rho)\ne 0$ by Lemma~\ref{lem:bessel_lemma}; so $\Lambda_{2, d}(\rho)>\Lambda_{4, d}(\rho)\ge \Lambda_{4+2h, d}(\rho)$, therefore $\mathcal{E}_d(\rho)=\mathcal H_2$. 
    
    The only remaining alternative is that $J_{\nu +2}(\rho)=0$; but then, both $J_{\nu +3}(\rho)\ne0$ and $J_{\nu +4}(\rho)\ne 0$ by Lemma \ref{lem:Bourget}. So
     \begin{equation}\label{eq:big_chain_i_stab}
        \begin{array}{ccc} 
            \Lambda_{0, d}(\rho)\!\!\!\! & >\Lambda_{2, d}(\rho) \!\!\!\!&> \Lambda_{4, d}(\rho)\ge \Lambda_{4+2h, d}(\rho) \\
            \vertbigspaced\!\!\!\! & \verteqspaced\!\!\!\! \\ 
            \Lambda_{1,d}(\rho)\!\!\!\!& =\Lambda_{3,d}(\rho)\!\!\!\!& > \Lambda_{5, d}(\rho) \ge \Lambda_{5+2h, d}(\rho)
        \end{array}
    \end{equation}
    and we conclude that 
    \begin{align}\label{eq:stab_i_concl}
        \mathbf{S}_d(\rho)&=\Lambda_{0, d}(\rho)-\Lambda_{1, d}(\rho)=\Lambda_{0, d}(\rho)-\Lambda_{2, d}(\rho)=\Lambda_{0, d}(\rho)-\Lambda_{3, d}(\rho), \\
        \mathcal{E}_d(\rho)&= \mathcal{H}_1\oplus \mathcal{H}_2 \oplus \mathcal{H}_3.
    \end{align}
     
     {\it{Case}} (ii): $(J_\nu J_{\nu +1})(\rho)<0$.  We have $\mathbf{C}_d(\rho)=\Lambda_{1, d}(\rho)$. By the inequalities~\eqref{eq:big_chain_ii}, 
    \begin{equation}\label{eq:S_d_case_ii}
        \mathbf{S}_d(\rho)=\Lambda_{1, d}(\rho)-\max_{k  {\in\{0, 3\}}} \Lambda_{k, d}(\rho),
    \end{equation}
    so we are led to define $\mathfrak{J}_{\nu}(\rho)$ as follows, where we also use the Bessel recursion~\eqref{eq:RR2}:
    \begin{equation}\label{eq:StabCaseiiSelector}
        \begin{split}
            \Lambda_{0, d}(\rho)-\Lambda_{3, d}(\rho) &=\Lambda_{0, d}(\rho)-\Lambda_{1, d}(\rho)+\Lambda_{1, d}(\rho)-\Lambda_{3, d}(\rho) \\
            &=(J_\nu J_{\nu+1})(\rho)+\frac2\rho(\nu+2)J_{\nu+2}^2(\rho)\\
            &{= (J_\nu J_{\nu+1})(\rho) +(J_{\nu +1}J_{\nu +2})(\rho) + ( {J_{\nu+2}J_{\nu+3}})(\rho)=}:\mathfrak{J}_{\nu}(\rho).
        \end{split}
    \end{equation}
    If $\mathfrak{J}_\nu(\rho)>0$ then ${\bf S}_d(\rho)=\Lambda_{1, d}(\rho)-\Lambda_{0, d}(\rho)$ and  {so}, since $J_{\nu+1}(\rho)\ne 0$, it follows that $\Lambda_{0, d}(\rho)>\Lambda_{2, d}(\rho)\ge \Lambda_{2+2h, d}(\rho)$, thus $\mathcal{E}_d(\rho)=\mathcal{H}_0$. 
    
    On the other hand, if $\mathfrak{J}_\nu(\rho)<0$ then ${\bf S}_d(\rho)=\Lambda_{1, d}(\rho)-\Lambda_{3, d}(\rho)$. {In this case,} we must have $\Lambda_{3, d}(\rho)>\Lambda_{5, d}(\rho)\ge \Lambda_{5+2h, d}(\rho)$. Indeed, assuming {towards} a contradiction that $\Lambda_{3, d}(\rho)=\Lambda_{5, d}(\rho)$, we would have $J_{\nu+4}(\rho)=0$ and so $(J_{\nu+2}J_{\nu+3})(\rho)>0$. But this contradicts $\mathfrak{J}_\nu(\rho)<0$; indeed, using the Bessel recursion {as before}, we see that
    \begin{equation}\label{eq:nontrivial_case_ii_stability}
        \begin{split}
             \mathfrak{J}_\nu(\rho)
            &=\frac{2}{\rho} {(\nu+1)}J_{\nu +1}^2(\rho)+(J_{\nu+2}J_{\nu +3})(\rho),
        \end{split}
    \end{equation}
    and the right-hand side is positive. We conclude that $\mathcal{E}_d(\rho)=\mathcal{H}_3$.
    
    Finally, if $\mathfrak{J}_\nu(\rho)=0$, then
    \begin{equation}\label{eq:stab_ii_concl}
        \mathbf{S}_d(\rho)=\Lambda_{1, d}(\rho)-\Lambda_{0, d}(\rho)=\Lambda_{1, d}(\rho)-\Lambda_{3, d}(\rho).
    \end{equation}
    Now, $\Lambda_{0, d}(\rho)>\Lambda_{2, d}(\rho)\ge \Lambda_{2+2h, d}(\rho)$, because $J_{\nu+1}(\rho)\ne 0$. {On the other hand, we have}  $\Lambda_{3, d}(\rho)>\Lambda_{5, d}(\rho)\ge \Lambda_{5+2h, d}(\rho)$ because $J_{\nu+4}(\rho)\ne 0$, which is proved by contradiction as we did in the previous paragraph. We conclude that $\mathcal{E}_d(\rho)=\mathcal{H}_0\oplus\mathcal{H}_3$. 
    
     {\it{Case}} (iii): $J_{\nu}(\rho)=0$. By the inequalities~\eqref{eq:big_chain_iii}, 
    \begin{equation}\label{eq:S_d_case_iii}
        \mathbf{S}_d(\rho)=\Lambda_{0, d}(\rho)-\max_{k  {\in\{2, 3\}}} \Lambda_{k, d}(\rho).
    \end{equation}
    We have that $\Lambda_{2, d}(\rho)-\Lambda_{3, d}(\rho)=(J_{\nu+2}J_{\nu+3})(\rho)$; if the latter is strictly positive, then ${\bf S}_d(\rho)=\Lambda_{0, d}(\rho)-\Lambda_{2, d}(\rho)$ and $\Lambda_{2, d}(\rho)>\Lambda_{4, d}(\rho)\ge \Lambda_{4+2h, d}(\rho)$ because $J_{\nu+3}(\rho)\ne 0$, so $\mathcal{E}_d(\rho)=\mathcal{H}_2$. If, on the other hand, $(J_{\nu+2}J_{\nu+3})(\rho)<0$, which implies $J_{\nu+4}(\rho)\ne 0$, then $  {\bf S}_d(\rho)=\Lambda_{0, d}(\rho)-\Lambda_{3, d}(\rho)$ and $\Lambda_{3, d}(\rho)>\Lambda_{5, d}(\rho)\ge \Lambda_{5+2h, d}(\rho)$, and we conclude that  $\mathcal{E}_d(\rho)=\mathcal{H}_3$. The case $(J_{\nu+2}J_{\nu +3})(\rho)=0$ does not occur, as it would contradict Lemma \ref{lem:Bourget}. 
    
     {\it{Case}} (iv): $J_{\nu +1}(\rho)=0$. By the inequalities~\eqref{eq:big_chain_iv}, \begin{equation}\label{eq:S_d_case_iv}
        {\bf S}_d(\rho)=\Lambda_{0, d}(\rho)-\max_{k  {\in\{3, 4\}}} \Lambda_{k, d}(\rho).
    \end{equation}
    The proof follows the exact same steps of the previous case upon replacing $\nu$ by $\nu+1$.
    
    The proof of Theorem~\ref{thm:stability} is complete.
    
\section{ {Regularity of ${\bf C}_d(\rho)$ and ${\bf S}_d(\rho)$: Proof of Corollary \ref{cor:regularity}}}\label{sec:regularity}
  {Recall} that $  {\{j_{\nu, k}\}_{k\geq 1}}$ denotes the sequence of positive zeroes of $J_\nu$; see Appendix   {\ref{app:Bessel}}. To prove Corollary~\ref{cor:regularity}, we start by showing that ${\bf C}_d(\rho)$ is not differentiable at $j_{\nu,k}$, where $k\geq 1$ is arbitrary. The argument for $j_{\nu+1,k}$ is entirely analogous.
From   {the proof of} Theorem~\ref{thm:main} and identity~\eqref{eq:recursion_odd}, respectively, we have that 
\begin{equation}\label{eq:prepare_proof_Cd_nondiff}
    \begin{array}{ccc}
        {\bf C}_d(\rho)=\max\{\Lambda_{0,d}(\rho),\Lambda_{1,d}(\rho)\}& \text{and} & \Lambda_{0,d}(\rho)-\Lambda_{1,d}(\rho)=(J_\nu J_{\nu+1})(\rho).
    \end{array}
\end{equation}
If ${\bf C}_d(\rho)$ were differentiable at $\rho=j_{\nu,k}$, then necessarily $\Lambda_{0,d}'(j_{\nu,k})=\Lambda_{1,d}'(j_{\nu,k})$. Instead,
\begin{equation}\label{eq:non_diff_Cd}
    \Lambda'_{0,d}(j_{\nu,k})-\Lambda'_{1,d}(j_{\nu,k})=J'_\nu(j_{\nu,k}) J_{\nu+1}(j_{\nu,k})+J_\nu(j_{\nu,k}) J'_{\nu+1}(j_{\nu,k}) =J'_\nu(j_{\nu,k}) J_{\nu+1}(j_{\nu,k})\neq 0.
\end{equation}
Indeed, $J'_\nu(j_{\nu,k})\neq 0$ since all the zeroes of $J_\nu$ are simple, and $J_{\nu+1}(j_{\nu,k})\neq 0$ in light of Lemma \ref{lem:Bourget}.

We now show that ${\bf C}_d(\rho)=\max\{\Lambda_{0,d}(\rho),\Lambda_{1,d}(\rho)\}$ defines a Lipschitz function on the positive half-line $  {(}0,\infty)$. Since the maximum of two Lipschitz functions is Lipschitz, it will suffice to show that $\Lambda_{0,d}(\rho)$ is Lipschitz; the proof for $\Lambda_{1,d}(\rho)$ is entirely analogous. In fact, the derivative of \begin{equation}\label{eq:LambdaZero}
    \Lambda_{0,d}(\rho)=\frac\rho 2 J_{\nu}^2(\rho) - \frac \rho 2 (J_{\nu-1}J_{\nu+1})(\rho)
\end{equation}
is uniformly bounded on $  {(}0,\infty)$. To prove this, we invoke the Bessel recursion~\eqref{eq:RR1} to   {obtain}, for $\rho>0$, 
\begin{equation}\label{eq:LambdaPrime}
\Lambda_{0,d}'(\rho)=\frac12(J_\nu^2-J_{\nu-1}J_{\nu+1})(\rho)+\frac \rho4({J_{\nu-1}J_{\nu}-J_{\nu-2}J_{\nu+1}-J_{\nu}J_{\nu+1}+J_{\nu-1}J_{\nu+2}})(\rho).
\end{equation}
In particular,   {a direct computation shows that} the right derivative at 0 satisfies $\Lambda_{0,2}'(0^+)=\frac12$, and  $\Lambda_{0,d}'(0^+)=0$, for   {$d\in\{3,4,5\}$}.   {For $d\geq 6$, the fact that $\Lambda_{0,d}'(0^+)=0$ follows at once from the behaviour of the Bessel functions at the origin, which in turn can be read off from \eqref{eq:BesselTaylor}}.
On the other hand, the asymptotic~\eqref{eq:bessel_asympt} of the Bessel functions at infinity guarantees that~\eqref{eq:LambdaPrime} remains bounded, as $\rho\to\infty$. 

Finally,   {recall that the zeroes of $J_\nu$ and $J_{\nu+1}$ interlace}. On each interval $(j_{\nu,k},j_{\nu+1,k})$, resp. $(j_{\nu+1,k},j_{\nu,k+1})$,  we have that ${\bf C}_d(\rho)$ equals $\Lambda_{0,d}(\rho)$, resp. $\Lambda_{1,d}(\rho)$, and so the claimed real-analyticity follows from identity \eqref{eq:Lommel}.

Now we turn to the analysis of ${\bf S}_d(\rho)$.
Start by noting that, from   {the proof of} Theorem \ref{thm:stability}, it follows that ${\bf S}_d(\rho)>0$, for every $\rho>0$. 
In order to show that ${\bf S}_d(\rho)$ has a jump discontinuity at each positive zero of $J_\nu J_{\nu+1}$, it suffices to check that 
\begin{equation}\label{eq:Limits}
\lim_{\rho\to j_{\nu,k}} {\bf S}_d(\rho)=0=\lim_{\rho\to j_{\nu+1,k}} {\bf S}_d(\rho),
\end{equation}
whenever $k\geq 1$.
We verify the first identity in \eqref{eq:Limits}, since the second one can be dealt with in an analogous way.
Choose $k\geq 1$, set $\rho^\star:=j_{\nu,k}$, and note that from Theorem~\ref{thm:main} it follows that $\mathcal M_d(\rho)\subsetneq \mathcal M_d(\rho^\star)$, whenever $\rho\neq \rho^\star$ is sufficiently close to $\rho^\star$.
Let $f_\star\in\mathcal M_d(\rho^\star)\setminus \mathcal M_d(\rho)$, for every such $\rho\neq \rho^\star$. By Theorem~\ref{thm:stability},
\begin{equation}\label{eq:proof_Sd_jumps}
    0\leq {\bf S}_d(\rho) \textup{d}^2(f_\star,\mathcal M_d(\rho)) \leq \delta_d(f_\star;\rho)=\mathbf{C}_d(\rho)\lVert f_\star\rVert_{L^2(\mathbb S^{d-1})}^2-\frac{1}{\rho}\int_{B_\rho} \lvert \widehat{f_\star\sigma}  {(x)}\rvert^2\frac{  {\d} x}{(2\pi)^d}.
\end{equation}
In the last equation we recalled the definition of $\delta_d(f_\star;\rho)$ to facilitate the proof of its continuity in   {the variable} $\rho>0$; indeed, we already proved that $\mathbf{C}_d(\rho)$   {defines a} continuous   {function of $\rho$}, and the other  {(integral)} term is also seen to be continuous  {in $\rho$} by dominated convergence. So $\delta_d(f_\star;\rho)\to 0$, as $\rho\to \rho^\star$, 
and since $\textup{d}^2(f_\star,\mathcal M_d(\rho))\neq 0$ is independent of $\rho$ for all $\rho\neq \rho^\star$ under consideration, identity \eqref{eq:Limits} follows.
This establishes the first claim about ${\bf S}_d(\rho)$. The second claim  follows similarly to the corresponding claim about ${\bf C}_d(\rho)$, with an entirely analogous proof, which is therefore omitted.

We finish by showing that ${\bf S}_d(\rho)$ is not differentiable at $j_{\nu+2,k}$, for arbitrary $k\geq 1$. 
If $\rho$ is sufficiently close to $j_{\nu+2,k}$, then necessarily $(J_\nu J_{\nu+1})(\rho)>0$. This follows by continuity from $(J_\nu J_{\nu+1})(j_{\nu+2,k})>0$, which in turn is a consequence of Lemma~\ref{lem:bessel_lemma}.
By  {C}ase (i) of Theorem~\ref{thm:stability}  and identity~\eqref{eq:recursion_odd} respectively, we then have that
\begin{equation}\label{eq:non_diff_Sd}
    \begin{array}{ccc}
        {\bf S}_d(\rho)=\Lambda_{0,d}(\rho)-\max\{\Lambda_{1,d}(\rho),\Lambda_{2,d}(\rho)\}& \text{and} & \Lambda_{1,d}(\rho)-\Lambda_{2,d}(\rho)=(J_{\nu+1}J_{\nu+2})(\rho).
    \end{array}
\end{equation}
The rest of the argument is similar to what we did before:
If ${\bf S}_d(\rho)$ were differentiable at $j_{\nu+2,k}$, then necessarily $\Lambda'_{1,d}(j_{\nu+2,k})=\Lambda'_{2,d}(j_{\nu+2,k})$, but instead we have
\[\Lambda'_{1,d}(j_{\nu+2,k})-\Lambda'_{2,d}(j_{\nu+2,k})
=J_{\nu+1}(j_{\nu+2,k})J_{\nu+2}'(j_{\nu+2,k})\neq 0,\]
in light of Lemma \ref{lem:Bourget} and the simplicity of the zeroes of $J_{\nu+2}$. 

This concludes the proof of Corollary \ref{cor:regularity}.

\section{The $\rho\to \infty$ limit  {of the Agmon--H\"ormander estimate}}
\label{sec:limit}

The following  {result} is the case of~\cite[Theorem~3.1]{AgHo76}  {which} is relevant to the present paper. To facilitate the reading, we provide here an adaptation of the original proof. We will then conclude by recovering the same result via the method of the previous sections.
\begin{theorem}\label{thm:AgHorLimit}
    For each $f\in L^2(\mathbb S^{d-1})$,
    \begin{equation}\label{eq:AgHorLimit_section}
        \lim_{\rho\to \infty} \frac1\rho \int_{B_\rho}\lvert \widehat{f\sigma}(x)\rvert^2\, \frac{\d x}{(2\pi)^d}=\frac1\pi \lVert f\rVert_{L^2  {(\mathbb S^{d-1})}}^2.
    \end{equation}
\end{theorem}
\begin{proof}
    Let $\mathbbm{1}_B$ denote the indicator function of $B:=B_1$, the unit ball   {in $\mathbb R^d$} centred at the origin. By Plancherel's Theorem, 
    \begin{equation}\label{eq:StraightPlancherel}
        \frac1\rho\int_{B_\rho} \lvert \widehat{f\sigma}(x)\rvert^2\, \frac{\d x}{(2\pi)^d} =\rho^{d-1}\iint_{(\mathbb S^{d-1})^2}\!\! f(\xi)\overline{f(\eta)} \widehat{\mathbbm{1}_B}(\rho(\xi-\eta))\, \frac{\d\sigma(\xi)\d\sigma(\eta)}{(2\pi)^d}.
    \end{equation}
    We will use the formula  $\d\sigma(\xi)=2\ddirac{1-\lvert\xi\rvert^2}\, \d\xi$, where $\ddirac{\cdot}$ denotes the one-dimensional Dirac distribution; see, for example, the appendix to~\cite{FoOl17} for more details on this and other formulae of this kind. Before letting $\rho\to \infty$, we apply the change of variables 
    \begin{equation}\label{eq:varchange}
        \begin{array}{rcl}
            \begin{cases}
                    y=\rho(\xi-\eta), \\
                    z=\eta,
            \end{cases}
            &
            \text{so}& \displaystyle
            \d\sigma(\xi)\d\sigma(\eta)
            %=2\delta(1-\lvert\xi\rvert^2)d\xi d\sigma(\eta)
            =\frac{2}{\rho^d}\ddirac{ 1-\left\lvert z-\frac y \rho \right\rvert^2}\d y\d\sigma(z).
        \end{array}
    \end{equation}
    We also observe that, for $\lvert z\rvert=1$, \begin{equation}\label{eq:DeltaManipulation}
        \ddirac{1-\left\lvert z-\frac y \rho\right\rvert^2} = \frac \rho2\ddirac{z\cdot y -\frac{\lvert y \rvert^2}{2\rho}}.
    \end{equation}
     {Thus} we see that~\eqref{eq:StraightPlancherel} equals
    \begin{equation}\label{eq:PassLimit}
        \begin{split}
            & \frac{2}{\rho}\iint_{\mathbb R^d\times \mathbb S^{d-1}} f\left(z+\frac y \rho\right) \overline{f(z)} \widehat{\mathbbm{1}_B}(y)\ddirac{1-\left\lvert z-\frac y \rho\right\rvert^2}\, \frac{\d y\d\sigma(z)}{(2\pi)^d} \\ 
            & \longrightarrow  \int_{\mathbb S^{d-1}}\lvert f(z)\rvert^2\left(\int_{\mathbb R^d} \widehat{\mathbbm{1}_B}(y)\ddirac{z\cdot y}\, \frac{\d y}{(2\pi)^d}\right)\, \d\sigma(z),\text{ as }\rho\to\infty.
        \end{split}
    \end{equation}
    We conclude by evaluating the   {latter inner} integral. Since $\mathbbm{1}_B$ is radially symmetric, that integral is independent on $z\in\mathbb S^{d-1}$, so we assume that $z=(0, \ldots,0, 1)$ and obtain  
    \begin{equation}\label{eq:Radon_transform}
        \begin{split}
            \int_{\mathbb R^d} \widehat{\mathbbm{1}_B}(y)\ddirac{z\cdot y}\, \frac{\d y}{(2\pi)^d} &= \int_{\mathbb R^{d-1}} \widehat{\mathbbm{1}_B}(y_1, \ldots, y_{d-1}, 0) \frac{\d y_1\ldots \d y_{d-1}}{(2\pi)^d} \\ 
            &= \int_{-\infty}^\infty \mathbbm{1}_B(0, \ldots, 0, \xi_d)\,\frac{ \d\xi_d}{2\pi} =\frac1\pi.
        \end{split}
    \end{equation}
    The proof is complete.
\end{proof}
As we saw in  {\S}\ref{sec:ProofThm1}, Theorem~\ref{thm:AgHorLimit} is equivalent to the statement that 
\begin{equation}\label{eq:Lambda_Limit}
    \begin{array}{cc}\displaystyle
        \lim_{\rho\to \infty} \Lambda_{k, d}(\rho)=\frac1\pi, &   {\text{for every}} \,k \in \mathbb N_{\ge 0}.
    \end{array}
\end{equation}
We give a direct proof of this statement. Applying the Bessel recursion~\eqref{eq:RR2}, we obtain the following alternative expression for $\Lambda_{k, d}(\rho)$:
\begin{equation}\label{eq:LambdaAlternative}
    \begin{split}
        \Lambda_{k, d}(\rho)
        &=\frac{\rho}{2}\left( J_{\nu+k}^2(\rho)-\frac{2(\nu+k)}{\rho} (J_{\nu+k}J_{\nu+k+1})(\rho) + J_{\nu+k+1}^2(\rho)\right).
    \end{split}
\end{equation}
By the Bessel asymptotic~\eqref{eq:bessel_asympt}, we see that
\begin{equation}\label{eq:Lambda_k_d_asymptotic}
    \begin{split}
        \Lambda_{k, d}(\rho)&=\frac1\pi\left( \cos^2\left( \rho-\frac{2\nu+2k+1}{4}\pi\right) + \cos^2\left( \rho-\frac{2\nu+2k+3}{4}\pi\right)\right) +O\left(\rho^{-1}\right) \\
        &=\frac1\pi \left( \cos^2\left( \rho-\frac{2\nu+2k+1}{4}\pi\right) + \sin^2\left( \rho-\frac{2\nu+2k+1}{4}\pi\right)\right) +O\left(\rho^{-1}\right) \\
        &=\frac1\pi + O\left(\rho^{-1}\right),
    \end{split}
\end{equation}
which proves~\eqref{eq:Lambda_Limit}.

\section*{Acknowledgements}
GN and DOS are supported by the EPSRC New Investigator Award ``Sharp Fourier Restriction Theory'', grant no.\@ EP/T001364/1.
DOS acknowledges partial support from the Deutsche Forschungsgemeinschaft under Germany's Excellence Strategy -- EXC-2047/1 -- 390685813.
The authors thank the anonymous referee for carefully reading the manuscript and valuable suggestions.

\appendix
\section{Bessel Functions}\label{app:Bessel}
Bessel functions can be defined in a number of ways.
We follow the classical treatise \cite{Wa44} and define, for $\alpha>-1$ and $\Re(z)>0$, 
\begin{equation}\label{eq:BesselTaylor}
J_\alpha(z)=\left(\frac z2\right)^\alpha \sum_{n=0}^\infty \frac{(-1)^n(\frac z2)^{ 2 n}}{n!\Gamma(\alpha+n+1)}.
\end{equation}
When $\alpha\ge 0$, since $\Gamma(\alpha+n+1)\ge \Gamma(\alpha+1)$ for $n\in\mathbb N_{\ge 0}$, we have the {crude} bound
\begin{equation}\label{eq:crude_bound}
    \begin{array}{cc}\displaystyle
        \lvert J_\alpha(z)\rvert \le \left( \frac{\lvert z \rvert}{2}\right)^\alpha \frac{e^{\lvert z \rvert^2/4}}{\Gamma(\alpha +1)}, &\displaystyle \text{ so in particular } \lim_{\alpha \to \infty} J_\alpha(z)=0.
    \end{array}
\end{equation}

The Bessel function $J_\alpha$ satisfies the following recursion relations: 
\begin{align}
   J_{\alpha-1}(z)+J_{\alpha+1}(z)&=\frac{2\alpha}z J_\alpha(z), \label{eq:RR2} \\
   J_{\alpha-1}(z)-J_{\alpha+1}(z)&=2J_\alpha'(z).\label{eq:RR1}
\end{align}
As a consequence of \eqref{eq:RR2}, we have the following simple fact, which is used several times throughout the text. The proof is immediate. 
\begin{lemma}\label{lem:bessel_lemma}
    If  {$\alpha,z>0$} and $J_{  {\alpha+1}}(z)=0$, then $(J_{  {\alpha-1}}J_{  {\alpha}})(z)>0$.
\end{lemma}
For any fixed $\alpha\geq 0$, {and $r>0$}, one has the following asymptotic at infinity:
\begin{equation}\label{eq:bessel_asympt}
    J_\alpha(r)=\left(\frac{\pi r}2\right)^{-\frac12} \cos\left( r-\frac{2\alpha +1}{4}\pi\right) +O\left({r^{-\frac32}}\right),  \text{ as } r\to\infty.
\end{equation}
The Bessel function $J_\alpha$ is entire if $\alpha$ is an integer, otherwise it is a multivalued function with a singularity at the origin. 
However, when $\alpha$ is an half-integer, then $J_\alpha$ is an elementary function, and one easily checks that it is real-analytic on the positive half-line $(0,\infty)$.

Finally we discuss zeroes of Bessel functions.
The function $J_\alpha$ has infinitely many positive zeros, all of which are simple, isolated, and denoted by
\[0<j_{\alpha,1}<j_{\alpha,2}<j_{\alpha,3}<\ldots\]
  {The zeroes $\{j_{\alpha,k}\}_{k\geq 1}$ and $\{j_{\alpha+1,k}\}_{k\geq 1}$ are well-known to interlace. 
  The following result is known as {\it Bourget hypothesis}; see \cite[\S 15.28]{Wa44}.
  \begin{lemma}[Bourget hypothesis]\label{lem:Bourget}
  Let $\alpha\geq 0$ be rational and  $m \geq 1$ be an integer.
  Then the functions $J_\alpha(r)$ and $J_{\alpha+ m}(r)$ have no common zeros other than the one at $r = 0$. 
  \end{lemma}
The reader will have noticed that, for the purposes of the present paper, only the considerably easier cases $m\in\{1,2,3,4\}$ of Lemma \ref{lem:Bourget} were needed.

\end{document}